\documentclass[11pt]{amsart}
\usepackage[margin=1in]{geometry}

\usepackage{amssymb}
\usepackage{amsthm}
\usepackage{amsmath}
\usepackage{mathrsfs}
\usepackage{amsbsy}
\usepackage[all]{xy}
\usepackage{bm}
\usepackage{hyperref}
\usepackage{tikz}
\usepackage{array}
\usepackage{float}
\usepackage{enumerate}
\usepackage{xcolor}
\usepackage{hhline}
\allowdisplaybreaks
\usepackage[noadjust]{cite}

\newenvironment{enumerate*}%
  {\begin{enumerate}[(I)]%
    \setlength{\itemsep}{10pt}%
    \setlength{\parskip}{0pt}}%
  {\end{enumerate}}

\newtheorem{theorem}{Theorem}[section]
\newtheorem{proposition}[theorem]{Proposition}

\newtheorem{lemma}[theorem]{Lemma}

\theoremstyle{definition}

\begin{document}

\title[]{The Runsort Permuton}
\subjclass[2010]{}

\author[Noga Alon, Colin Defant, and Noah Kravitz]{Noga Alon}
\address{Department of Mathematics, Princeton University, Princeton, NJ 08544, USA and Schools of Mathematics and Computer Science, Tel Aviv University, Tel Aviv 69978, Israel}
\email{nalon@math.princeton.edu}
\author[]{Colin Defant}
\address[]{Department of Mathematics, Princeton University, Princeton, NJ 08540, USA}
\email{cdefant@princeton.edu}
\author[]{Noah Kravitz}
\address[]{Department of Mathematics, Princeton University, Princeton, NJ 08540, USA}
\email{nkravitz@princeton.edu}

\begin{abstract}
Suppose we choose a permutation $\pi$ uniformly at random from $S_n$. Let $\mathsf{runsort}(\pi)$ be the permutation obtained by sorting the ascending runs of $\pi$ into lexicographic order. Alexandersson and Nabawanda recently asked if the plot of $\mathsf{runsort}(\pi)$, when scaled to the unit square $[0,1]^2$, converges to a limit shape as $n\to\infty$. We answer their question by showing that the measures corresponding to the scaled plots of these permutations $\mathsf{runsort}(\pi)$ converge with probability $1$ to a permuton (limiting probability distribution) that we describe explicitly.  In particular, the support of this permuton is $\{(x,y)\in[0,1]^2:x\leq ye^{1-y}\}$.
\end{abstract}

\maketitle

\section{Introduction}

Let $S_n$ denote the set of permutations of the set $[n]=\{1,\ldots,n\}$. The \emph{scaled plot} of a permutation $\pi=\pi_1\cdots\pi_n\in S_n$ is the diagram showing the points $(i/n,\pi_i/n)$ for $i\in[n]$. The scaled plot of $\pi$ is closely related to the probability measure $\gamma_\pi$ on the unit square $[0,1]^2$ defined as follows. 
Consider a point $(x,y)\in[0,1]^2$. If $(i-1)/n\leq x<i/n$ and $(\pi_i-1)/n\leq y<\pi_i/n$ for some $i\in[n]$, then $\gamma_\pi$ has density $n$ at $(x,y)$; otherwise, $\gamma_\pi$ has density $0$ at $(x,y)$.
In other words, we divide $[0,1]^2$ into an $n\times n$ grid of squares, each with side length $1/n$, and we assign each square a constant density of either $n$ or $0$, according to whether or not the upper right corner of the square is a point in the scaled plot of $\pi$.

\begin{figure}[b]
  \begin{center}\includegraphics[height=5.7cm]{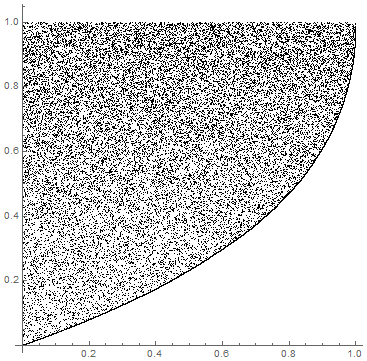}\qquad\qquad\qquad \includegraphics[height=5.5cm]{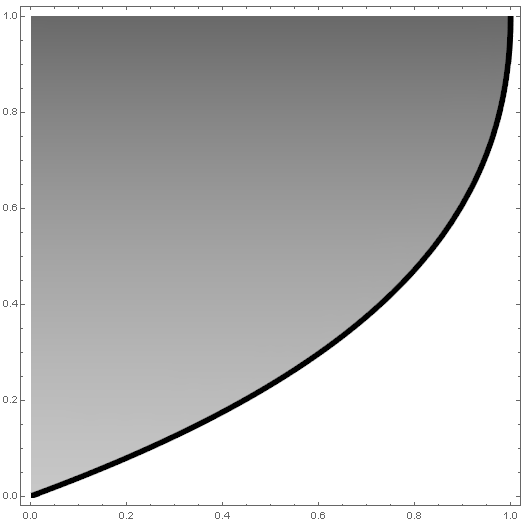}
  \end{center}
  \caption{On the left is the scaled plot of $\mathsf{runsort}(\pi)$, where $\pi$ is a permutation chosen uniformly at random from $S_{50000}$. On the right is the runsort permuton ${\bf R}$. The dark curve $\mathcal C$ is the support of the singular continuous part of ${\bf R}$. Shading within the region $\mathcal C^+$ indicates the density of the absolutely continuous part of ${\bf R}$.}\label{Fig1}
\end{figure}

A \emph{permuton} is a probability measure $\gamma$ on the unit square $[0,1]^2$ that has uniform marginals, in the sense that $\gamma([a,b]\times[0,1])=\gamma([0,1]\times[a,b])=b-a$ for all $0\leq a\leq b\leq 1$. Note that if $\pi\in S_n$, then $\gamma_\pi$ is a permuton.  (This is why we chose to scale with density $n$.) There is a natural topology on the space of permutons obtained by restricting the weak topology on probability measures. This coincides with the topology induced by the metric $\text{d}_\square$ defined by \[\text{d}_{\square}(\gamma_1,\gamma_2)=\sup_{B}\lvert\gamma_1(B)-\gamma_2(B)\rvert,\] where we take the supremum over all axis-parallel rectangles in the unit square. There has been a great deal of recent interest in describing permutons that arise as limits of large random permutations. Several results in this area concern the interplay between random permutations and permutation patterns \cite{Atapour, Dokos, Glebov, Hoppen, Kenyon, Miner}. In a different direction, Dauvergne \cite{Dauvergne} recently proved a beautiful result about permutons arising from random sorting networks (see also \cite{Angel, DauvergneVirag}).
Our goal in this article is to describe the permuton that emerges when we apply an operator called $\mathsf{runsort}$ to a large random permutation.

An \emph{ascending run} of a permutation (henceforth called a \emph{run} for simplicity) is a maximal consecutive increasing subsequence. For instance, the runs of $351476298$ are $35$, $147$, $6$, $29$, and $8$. Given a permutation $\pi\in S_n$, let $\mathsf{runsort}(\pi)$ be the permutation obtained by sorting the runs of $\pi$ into lexicographic order. Equivalently, $\mathsf{runsort}$ sorts the runs so that their smallest entries appear in increasing order. For example, $\mathsf{runsort}(351476298)=147293568$. Note that $\mathsf{runsort}$ is an idempotent operator. 

Motivated by the study of flattened partitions (see \cite{Callan, Mansour, Mansour2, Nabawanda}), Alexandersson and Nabawanda \cite{peaks} proved several interesting combinatorial properties of $\mathsf{runsort}$. When they chose $\pi\in S_{n}$ (for large $n$) uniformly at random and plotted the permutation $\mathsf{runsort}(\pi)$, they observed that it tended to have a very distinctive shape (see the left side of Figure~\ref{Fig1}). Furthermore, they noticed that the scaled plot of $\mathsf{runsort}(\pi)$ appeared to be bounded by a certain enveloping curve, and they asked if this curve approaches some limit curve as $n\to\infty$. In this paper, we answer their question (in a strong form) by showing that $\gamma_{\mathsf{runsort}(\pi)}$ converges with high probability to a specific permuton.

Consider the curve \[\mathcal C=\{(x,y)\in[0,1]^2:x=ye^{1-y}\},\] and let \[\mathcal C^+=\{(x,y)\in[0,1]^2:x<ye^{1-y}\}\] be the region in the unit square above $\mathcal C$.    We define the \emph{runsort permuton} ${\bf R}$ to be the permuton given by 
\[{\bf R}(B)=\int_{B\cap\,\mathcal C^+}e^{y-1}\,dx\,dy+\int_{B\cap\,\mathcal C}(1-y)\,dy\]
for every measurable set $B\subseteq [0,1]^2$.  Thus, ${\bf R}$ is the sum of its absolutely continuous part, which has support $\mathcal C^+$, and its singular continuous part, which has support $\mathcal C$.  One can easily confirm by direct computation that ${\bf R}$ is in fact a permuton (and we encourage the reader to do this).

Our main result is that ${\bf R}$ is the limiting distribution for the image under $\mathsf{runsort}$ of a large random permutation.

\begin{theorem}[Main Theorem]\label{thm:main}
Fix any $\varepsilon>0$, and choose $\pi^{(n)}$ uniformly at random from $S_n$.  Then $d_{\square}(\gamma_{\mathsf{runsort(}\pi^{(n)})},{\bf R})<\varepsilon$ with probability tending to $1$ as $n \to \infty$.
\end{theorem}

In other words, if we randomly choose permutations $\pi^{(n)}$, then the measures $\gamma_{\mathsf{runsort}(\pi^{(n)})}$ converge to $\bf R$ with probability $1$.

We remark that a combination of absolutely continuous and singular continuous parts, such as what ${\bf R}$ exhibits, seems to be rare in previous work on permutons. In our setting, however, it is quite natural. As will become clear later, the singular part comes from the first entries of the runs, and the continuous part comes from the remaining entries.  A random permutation in $S_n$ has $(n+1)/2$ runs on average, and the reader can check that indeed half of the total mass of ${\bf R}$ lies on the curve $\mathcal{C}$.  More surprising is that the pointwise density of ${\bf R}$ in $\mathcal{C}^+$ depends only on the vertical coordinate $y$.  This ``horizontal uniformity'' is not combinatorially obvious, and we do not know how to establish it directly (that is, without recourse to the explicit characterization of ${\bf R}$).

In Section~\ref{sec:martingale}, we use martingale concentration inequalities to show that it suffices to understand the expectation of $\gamma_{\mathsf{runsort}(\pi^{(n)})}$.  In Section~\ref{sec:enveloping}, we treat the enveloping curve $\mathcal{C}$ and determine the distribution of mass on it.  Section~\ref{sec:interior} is devoted to the density in the region $\mathcal{C}^+$.  Section~\ref{sec:everything} concludes the proof of Theorem~\ref{thm:main}.  Finally, in Section~\ref{sec:conclusion}, we briefly mention a possible generalization of this type of question where our methods could be adapted.

\subsection{Notation}
We write $\pi_i$ for the $i$-th entry of a permutation $\pi$. We write $\mathbb P[A]$ for the probability of an event $A$ and write $\mathbb E[X]$ for the expected value of a random variable $X$. 

\section{Concentration of Distribution}\label{sec:martingale}
Our task is to show that if $\pi$ is chosen randomly from $S_n$, then $\gamma_{\mathsf{runsort}(\pi)}$ is suitably concentrated everywhere.  The key observation is that transposing two entries of $\pi$ has only a small effect on $\gamma_{\mathsf{runsort}(\pi)}$, as long as $\pi$ does not have any unusually long runs.  In order to make this notion precise, we define the following variant of $\mathsf{runsort}$: Let $\pi \in S_n$ be a permutation, and let $A_1, \ldots, A_t$ be the runs of $\pi$.  We can break each run $A_k$ into shorter subsequences $A_{k,1}, \ldots, A_{k,t(k)}$, which we call \emph{segments}, as follows: If $A_k$ has length smaller than $\log n$, then set $t(k)=1$ and $A_{k,1}=A_k$; if $A_k$ has length greater than $\log n$, then break it directly before every position (of $\pi$) that is an integer multiple of $\lfloor \log n \rfloor$, so that $A_{k,\ell}$ has length exactly $\lfloor \log n \rfloor$ for $1<\ell<t(k)$, and $A_{k,1}$ and $A_{k,t(k)}$ have length at most $\lfloor \log n \rfloor$.  We then define $\overline{\mathsf{runsort}}(\pi)$ to be the permutation obtained by sorting the runs $A_{k,\ell}$ into lexicographic order.  Note that $\mathsf{runsort}(\pi)=\overline{\mathsf{runsort}}(\pi)$ if all of the runs of $\pi$ have length at most $\lfloor \log n \rfloor$; in particular, the following lemma tells us that if $\pi$ is chosen randomly from $S_n$, then $\mathsf{runsort}(\pi)=\overline{\mathsf{runsort}}(\pi)$ with probability tending very quickly to $1$ as $n \to \infty$.

\begin{lemma}\label{lem:loglog}
Suppose $\pi$ is chosen uniformly at random from $S_n$. If $n$ is sufficiently large, then the probability that every run of $\pi$ has length at most $\log n$ is at least $1-n^{-(\log\log n)/2}$. 
\end{lemma}

\begin{proof}
Let $T=\left\lfloor\log n\right\rfloor+1$. For each $i\in[n-T+1]$, the probability that the entries of $\pi$ in positions $i,i+1,\ldots,i+T-1$ appear in increasing order is $1/T!$. Therefore, the probability that there is a run of $\pi$ of length at least $T$ is at most $(n-T+1)/T!$, which, by Stirling's Formula, is at most $n^{-(\log\log n)/2}$ for $n$ sufficiently large. 
\end{proof}

\begin{lemma}\label{lem:transposition}
Let $B=[x_1, x_2] \times [y_1, y_2] \subset [0,1]^2$ be a rectangle.  Let $\pi \in S_n$ be a permutation, and let $\pi'=\pi \circ (i_1 \, i_2)$ be the permutation obtained from $\pi$ by swapping the entries in positions $i_1$ and $i_2$ (that is, applying the transposition $(i_1 \, i_2)$). Then $|\gamma_{\overline{\mathsf{runsort}}(\pi)}(B)-\gamma_{\overline{\mathsf{runsort}}(\pi')}(B)| \leq 20\log n/n$.
\end{lemma}

\begin{proof}
We note that $|\gamma_{\overline{\mathsf{runsort}}(\pi)}(B)-\gamma_{\overline{\mathsf{runsort}}(\pi')}(B)|$, viewed as a function of $x_1, x_2, y_1, y_2$, clearly attains its maximum value when $x_1, x_2, y_1, y_2$ are all integer multiples of $1/n$, so it suffices to prove the lemma when $B$ has this special form.  In this case, we discretize the problem by writing
\begin{equation}\label{eqn:transposition}
\gamma_{\overline{\mathsf{runsort}}(\pi)}(B)=\frac{\#\{i:x_1n<i \leq x_2n\text{ and } y_1n<\overline{\mathsf{runsort}}(\pi)_i \leq y_2n\}}{n}
\end{equation}
(since each point $(i, \overline{\mathsf{runsort}}(\pi)_i)$ contributes mass $1/n$ after rescaling), and likewise for $\gamma_{\overline{\mathsf{runsort}}(\pi')}(B)$.  When we apply the transposition $(i_1 \, i_2)$ to $\pi$, all but at most four of the segments $A_{k,\ell}$ remain the same.  Since each segment has length at most $\log n$, we see that at most $4 \log n$ entries are in these ``affected'' segments.  It follows that for each entry $j$ not in one of these affected runs, the horizontal positions of $j$ in $\overline{\mathsf{runsort}}(\pi)$ and $\overline{\mathsf{runsort}}(\pi')$ differ by at most $4 \log n$.  In particular, these horizontal positions are either both in the interval $(x_1n, x_2n]$ or neither in this interval, unless the position in $\overline{\mathsf{runsort}}(\pi)$ was within $4\log n$ of either $x_1n$ or $x_2n$.  So the change in the numerator on the right-hand side of \eqref{eqn:transposition} coming from these entries is at most $16 \log n$, and we must absorb an additional error of $4 \log n$ because we have no control over the positions of the entries of the affected segments.
\end{proof}

(It is possible to replace the constant $20$ by $4$ in this lemma, but we do not optimize this constant because it is irrelevant in what follows.)

We require the following martingale inequality from \cite{maurey} (see also page 35 of the book \cite{Milman}).

\begin{proposition}[\cite{maurey, Milman}]\label{prop:martingale}
Let $f:S_n \to \mathbb{R}$ be a function on permutations such that if $\pi, \pi' \in S_n$ differ by a transposition, then $|f(\pi)-f(\pi')| \leq z$.  Then for $\pi \in S_n$ chosen uniformly at random, we have
$$\mathbb{P}[|f(\pi)-\mathbb{E}[f]| \geq cz] \leq 2e^{-c^2 /(4n)}.$$
\end{proposition}

We note that by following the standard proof of Azuma's Inequality (see, e.g., \cite[Theorem~7.2.1]{AS}) with the obvious modification needed to deal with permutations, it is possible to improve the above exponent from $-c^2/(4n)$ to $-c^2/(2n)$, but this is not necessary for our applications.

Before applying this proposition, we fix a bit of notation.  For $1 \leq i,j \leq n$, let $p_n(i,j)$ denote the probability that the entry $j$ is in the $i$-th position of $\mathsf{runsort}(\pi)$ when $\pi \in S_n$ is chosen uniformly at random.  For $B=[x_1, x_2] \times [y_1, y_2] \subseteq [0,1]^2$, let $$E_n(B)=\frac{1}{n}\sum_{\substack{x_1n<i \leq x_2n,\\ y_1n<j \leq y_2n}} p_n(i,j).$$  Define the analogous quantities $\overline{p}_n(i,j)$ and $\overline{E}_n(B)$ with $\mathsf{runsort}$ replaced by $\overline{\mathsf{runsort}}$.

\begin{lemma}\label{lem:concentration}
Fix any small $\varepsilon>0$, and suppose that $\pi \in S_n$ is chosen uniformly at random.  If $n$ is sufficiently large (as a function of $\varepsilon$), then with probability at least $1-\varepsilon$, the concentration inequality
$$|\gamma_{\mathsf{runsort}(\pi)}(B)-E_n(B)|<\varepsilon$$
holds simultaneously for all axis-parallel rectangles $B \subseteq [0,1]^2$.
\end{lemma}
\begin{proof}
We make two reductions.  First, write $B=[x_1, x_2] \times [y_1, y_2]$ and define $B'=[x'_1, x'_2] \times [y'_1, y'_2]$ where $x'_1=\lfloor nx_1 \rfloor/n$, $x'_2=\lfloor nx_2 \rfloor/n$, $y'_1=\lfloor ny_1 \rfloor/n$, $y'_2=\lfloor ny_2 \rfloor/n$.  Then $E_n(B)=E_n(B')$ and $$|\gamma_{\mathsf{runsort}(\pi)}(B)-\gamma_{\mathsf{runsort}(\pi)}(B')| \leq 2/n.$$
Since this can be made smaller than $ \varepsilon/3$, it suffices to show that $$|\gamma_{\mathsf{runsort}(\pi)}(B)-E_n(B)|<2\varepsilon/3$$
for all $B$ whose vertices have coordinates that are integer multiples of $1/n$.  Note that in this case, the quantities $E_n(B)$ and $\overline{E}_n(B)$ are precisely the expected values of $\gamma_{\mathsf{runsort}(\pi)}(B)$ and $\gamma_{\overline{\mathsf{runsort}}(\pi)}(B)$, respectively.

Second, we know from Lemma~\ref{lem:loglog} that $\mathsf{runsort}(\pi)=\overline{\mathsf{runsort}}(\pi)$ with probability at least $1-n^{-(\log \log n)/2}$.  In particular, $|p_n(i,j)-\overline{p}_n(i,j)| \leq n^{-(\log \log n)/2}$ for all $i,j$, so $|E_n(B)-\overline{E}_n(B)| \leq n \cdot n^{-(\log \log n)/2}$.  We can make this last quantity smaller than $\varepsilon/3$ by choosing $n$ large enough, so it suffices to show that with probability at least $1-2\varepsilon/3$, the inequality $$|\gamma_{\overline{\mathsf{runsort}}(\pi)}(B)-\overline{E}_n(B)|<\varepsilon/3$$
holds simultaneously for all axis-parallel rectangles $B$ whose coordinates are integer multiples of $1/n$.

For each such $B$ with coordinates that are integer multiples of $1/n$, combining Lemma~\ref{lem:transposition} and Proposition~\ref{prop:martingale} (with $z=20 \log n/n$ and $c=\varepsilon n/60 \log n$) gives that
$$|\gamma_{\overline{\mathsf{runsort}}(\pi)}(B)-\overline{E}_n(B)|<\varepsilon/3$$
with probability at least $1-e^{-\frac{\varepsilon^2 n}{14400 (\log n)^2}}$.  A union bound gives that with probability at least
$$1-n^4e^{-\frac{\varepsilon^2 n}{14400 (\log n)^2}},$$
the above inequality holds simultaneously for all $B$ with coordinates that are integer multiples of $1/n$, and taking $n$ large guarantees that this probability is at least $1-2\varepsilon/3$, as needed.
\end{proof}

This lemma tells us that it will suffice to work with the expectations $E_n(B)$.  In particular, we have reduced Theorem~\ref{thm:main} to the following more manageable-looking statement.  For each $n$, define the probability measure ${\bf R}_n$ via ${\bf R}_n(B)=E_n(B)$ for all axis-parallel rectangles $B \subseteq [0,1]^2$.

\begin{theorem}\label{thm:main-reduction}
The measures ${\bf R}_n$ converge to ${\bf R}$ as $n$ goes to infinity.
\end{theorem}

\section{The Enveloping Curve}\label{sec:enveloping}

In this section, we address the ``singular'' behavior that comes from the first entries of runs; as mentioned in the Introduction, this corresponds to the mass in ${\bf R}$ that lies on the curve $\mathcal{C}$.

For $y \in [0,1]$ and $\pi \in S_n$, we let $L_{\pi}(y)$ denote the largest position $i \in [n]$ such that $\pi_i \leq yn$ if $y \geq 1/n$, and we define $L_\pi(y)=0$ if $y<1/n$.  In other words, $L_{\pi}(y)$ is the smallest $i$ such that all of the entries up to $yn$ appear in the first $i$ positions.  Note that the curve $x=L_\pi(y)$ is the lower envelope of the scaled plot of $\pi$.  We start by computing the expected value of $L_{\mathsf{runsort}(\pi)}(y)$ and $L_{\overline{\mathsf{runsort}}(\pi)}(y)$.

\begin{proposition}\label{prop:L-expectation}
Fix $y\in[0,1]$, and suppose that $\pi \in S_n$ is chosen uniformly at random.  Then $$\mathbb E[L_{\mathsf{runsort}(\pi)}(y)]= nye^{1-y}+O(\log n)$$
and $$\mathbb E[L_{\overline{\mathsf{runsort}}(\pi)}(y)]= nye^{1-y}+O(\log n).$$
\end{proposition}

\begin{proof}
We start with the first statement.  For $j\in [n]$, define the random variable $X_j$ by \[X_j(\pi)=\begin{cases} 0, & \mbox{if $j$ is not the beginning of a run of $\pi$} ; \\ k, & \mbox{if $j$ is the beginning of a run of length $k$ of $\pi$}. \end{cases}\] It follows from the definition of $\mathsf{runsort}$ that \[\sum_{j\leq yn}X_j(\pi)-m(\pi)<L_{\mathsf{runsort}(\pi)}(y)\leq\sum_{j\leq yn}X_j(\pi),\] where $m(\pi)$ is the maximum length of a run in $\pi$. We know from Lemma~\ref{lem:loglog} that $m(\pi)\leq\log n$ with probability at least $1-n^{-(\log\log n)/2}$. Since $m(\pi)\leq n$ for all $\pi\in S_n$,
we have \[\mathbb E[L_{\mathsf{runsort}(\pi)}(y)]\geq \sum_{j\leq yn}\mathbb E[X_j(\pi)]-(1-n^{-(\log\log n)/2})\log n-n\cdot n^{-(\log\log n)/2}.\] Hence, \[\mathbb E[L_{\mathsf{runsort}(\pi)}(y)]=\sum_{j\leq yn}\mathbb E[X_j(\pi)]+O(\log n).\]

Consider $j\in[n]$ with $j \leq yn$, and let $r$ be such that $\pi_r=j$.  We wish to estimate the probability that $X_j(\pi) \geq k+1$ for each $k \geq 0$.  If $k>\log n$, then $\mathbb P[X_j(\pi)\geq k+1]<n^{-(\log\log n)/2}$ by Lemma~\ref{lem:loglog} (and this contribution will turn out to be negligible).  To handle the case $k\leq\log n$, note that we have $X_j(\pi) \geq k+1$ if and only if the following both hold:
\begin{itemize}
\item Either $r=1$, or $2 \leq r \leq n-k$ and $\pi_{r-1}>j$;
\item $j<\pi_{r+1}< \pi_{r+2}<\cdots < \pi_{r+k}$.
\end{itemize}
Therefore, \[\mathbb{P}[X_j(\pi) \geq k+1] =\frac{1}{n \cdot k!} \cdot \frac{\binom{n-j}{k}}{\binom{n-1}{k}}+\frac{n-k-1}{n \cdot k!} \cdot \frac{\binom{n-j}{k+1}}{\binom{n-1}{k+1}}.\]
Note that the first term is at most $1/n$. To estimate the second term, we write
\begin{align*}
\frac{\binom{n-j}{k+1}}{\binom{n-1}{k+1}} &=\prod_{t=0}^k\frac{n-j-t}{n-1-t}\\
 &=\prod_{t=0}^k\left(1-\frac{j}{n}+\frac{n-j-jt}{n(n-1-t)}\right)\\
 &=\left(1-\frac{j}{n}+O\left(\frac{t+1}{n}\right)\right)^{k+1}\\
 &=(1-j/n)^{k+1}+O((k+1)^2/n),
\end{align*}
where the last bound uses the fact that $k \leq \log n$.
Combining these estimates yields
\begin{align*}
\mathbb{P}[X_j(\pi) \geq k+1] &=O(1/n)+\frac{1}{k!} (1-(k+1)/n)\left((1-j/n)^{k+1}+O((k+1)^2/n)\right)\\
 &=\frac{(1-j/n)^{k+1}}{k!}+O(1/n),
\end{align*}
again using $k\leq \log n$.  So
\begin{align*}
\mathbb{E}[X_j(\pi)] &=\sum_{k=0}^{n-1} \mathbb{P}[X_j(\pi) \geq k+1]\\
 &=\sum_{k=0}^{\left\lfloor\log n\right\rfloor} \left( \frac{(1-j/n)^{k+1}}{k!}+O(1/n) \right)+O\left(n \cdot n^{-(\log\log n)/2} \right)\\
 &=(1-j/n)(e^{1-j/n}+O(1/\left\lfloor\log n\right\rfloor!))+O(\log n/n)+O\left(n \cdot n^{-(\log\log n)/2} \right)\\
  &=(1-j/n)e^{1-j/n}+O(\log n/n).
\end{align*}
Summing over $j \leq yn$ gives
\begin{align*}
\sum_{j \leq yn} \mathbb{E}[X_j(\pi)] &=\sum_{j \leq yn} \left( (1-j/n)e^{1-j/n}+O(\log n/n) \right)\\
 &=n\int_{0}^{y}(1-t)e^{1-t} \,dt+O(1)+O(\log n)\\
  &=nye^{1-y}+O(\log n),
\end{align*}
and we conclude that $$\mathbb E[L_{\mathsf{runsort}(\pi)}(y)]= nye^{1-y}+O(\log n).$$

For the statement about $\mathbb E[L_{\overline{\mathsf{runsort}}(\pi)}(y)]$, note that $L_{\overline{\mathsf{runsort}}(\pi)}(y)=L_{\mathsf{runsort}(\pi)}(y)$ with probability at least $1-n^{-(\log \log n)/2}$ by Lemma~\ref{lem:loglog}; when these quantities do differ, they differ by at most $n$, and $n\cdot n^{-(\log \log n)/2}$ can be absorbed into the error term.
\end{proof}

In order to apply Proposition~\ref{prop:martingale}, we need an analogue of Lemma~\ref{lem:transposition}.  As in the previous section, it is more convenient to work with $\overline{\mathsf{runsort}}$.

\begin{lemma}\label{lem:transposition-2}
Fix $y \in [0,1]$.  Let $\pi \in S_n$ be a permutation, and let $\pi'=\pi \circ (i_1 \, i_2)$ be the permutation obtained from $\pi$ by swapping the entries in positions $i_1$ and $i_2$.  Then \[|L_{\overline{\mathsf{runsort}}(\pi)}(y)-L_{\overline{\mathsf{runsort}}(\pi')}(y)| \leq 9\log n.\]
\end{lemma}

\begin{proof}
Write $\{A_{k, \ell}\}$ and $\{A'_{k, \ell}\}$ for the sets of segments of $\pi$ and $\pi'$, respectively.  As in the proof of Lemma~\ref{lem:transposition}, we note that multiplying $\pi$ on the right by the transposition $(i_1 \, i_2)$ affects at most four of the segments $A_{k, \ell}$, which together contain at most $4 \log n$ entries.  In particular, for each entry $j$ not in one of these affected segments, the horizontal positions of $j$ in $\overline{\mathsf{runsort}}(\pi)$ and $\overline{\mathsf{runsort}}(\pi')$ differ by at most $4 \log n$.  Hence, the first $L_{\overline{\mathsf{runsort}}(\pi)}(y)+4 \log n$ entries of $\overline{\mathsf{runsort}}(\pi')$ certainly contain all of the entries up to $yn$ except for possibly some of the $4 \log n$ entries in the affected segments.  These missing small entries (if any exist) are contained in at most four segments $A'_{k,\ell}$ of $\pi'$, and these segments must appear in $\overline{\mathsf{runsort}}(\pi')$ directly after the last run that includes all smaller entries.  Since all runs have length at most $\log n$, we see that all of the entries up to $yn$ are among the first
$$L_{\overline{\mathsf{runsort}}(\pi)}(y)+4 \log n+\log n+4\log n=L_{\overline{\mathsf{runsort}}(\pi)}(y)+9 \log n$$
entries of $\overline{\mathsf{runsort}}(\pi')$; that is, $L_{\overline{\mathsf{runsort}}(\pi')}(y) \leq L_{\overline{\mathsf{runsort}}(\pi)}(y)+9 \log n$.  By the same argument, we have $L_{\overline{\mathsf{runsort}}(\pi)}(y) \leq L_{\overline{\mathsf{runsort}}(\pi')}(y)+ 9 \log n$.
\end{proof}

Following the same strategy as in the previous section, we obtain a concentration inequality for $L_{\mathsf{runsort}(\pi)}(y)$.

\begin{lemma}\label{lem:curve-concentration}
Fix $y \in [0,1]$, and suppose that $\pi \in S_n$ is chosen uniformly at random.  Then with probability at least $1-2n^{-(\log \log n)/2}$, we have the concentration inequality
$$|L_{\mathsf{runsort}(\pi)}(y)-nye^{1-y}|=O(\sqrt{n} (\log n)^2).$$
\end{lemma}

\begin{proof}
Applying Proposition~\ref{prop:martingale} with $z=9 \log n$ and $c=2\sqrt{n}\log n$ gives that with probability at least $1-2e^{-(\log n)^2}$, the quantity $L_{\overline{\mathsf{runsort}}(\pi)}(y)$ differs from its expectation by at most $18\sqrt{n} (\log n)^2$.  Plugging in the estimate from Proposition~\ref{prop:L-expectation}, we have that
$$|L_{\overline{\mathsf{runsort}}(\pi)}-nye^{1-y}|=O(\sqrt{n} (\log n)^2)$$
with probability at least $1-2e^{-(\log n)^2}$.  Since $\mathsf{runsort}(\pi)=\overline{\mathsf{runsort}}(\pi)$ with probability at least $1-n^{-(\log \log n)/2}$ (by Lemma~\ref{lem:loglog}), we see that in fact $|L_{\overline{\mathsf{runsort}}(\pi)}-nye^{1-y}|=O(\sqrt{n} (\log n)^2)$ holds with probability at least
\[1-2e^{-(\log n)^2}-n^{-(\log \log n)/2}>1-2n^{-(\log \log n)/2}. \qedhere\]
\end{proof}

When the entry $j=yn$ is the beginning of a run of $\pi$, the position of $j$ in $\mathsf{runsort}(\pi)$ is precisely $L_{\mathsf{runsort}(\pi)}(y)$.  Thus, the previous lemma tells us that in the scaled plot of $\mathsf{runsort}(\pi)$, the beginnings of runs cluster around the curve $\mathcal{C}$.  To make this precise, let $q_n(i,j)$ denote the probability that the entry $j$ is the beginning of a run of $\pi$ and is in the $i$-th position of $\mathsf{runsort}(\pi)$ when $\pi \in S_n$ is chosen uniformly at random.  (Here, $q_n(i,j)$ differs from $p_n(i,j)$ in that the former looks at only the first entry of each run and the latter looks at all entries.)  We obtain an estimate on the distribution of $q_n(i,j)$ for fixed $j$.

\begin{lemma}\label{lem:singular-part}
There exists a constant $C>0$ such that the following holds for all $y \in [0,1]$:
$$\left|\sum_{|i-nye^{1-y}| \leq C\sqrt{n} (\log n)^2}q_n(i,yn)-(1-y+1/n)\right|\leq 2n^{-(\log \log n)/2}$$
and
$$\sum_{|i-nye^{1-y}|>C\sqrt{n} (\log n)^2}q_n(i,yn)\leq 2n^{-(\log \log n)/2}.$$
\end{lemma}

\begin{proof}
Recall that the entry $yn$ is the beginning of a run of $\pi$ with probability $(n-yn+1)/n$. This implies that $\sum_{i \in [n]}q_n(i,yn)=1-y+1/n$.  Both statements now follow from Lemma~\ref{lem:curve-concentration}.
\end{proof}

Let us summarize in words what this lemma tells us about the contribution to the measures ${\bf R}_n$ (and eventually also $\bf R$) from the beginnings of runs: This contribution is concentrated close to the curve $\mathcal{C}$, and the ``weighting'' in the $y$-direction is $(1-y) \, dy$.  We will make these observations precise in Section~\ref{sec:everything}.

Finally, we record a version of this result that will be convenient in the next section.

\begin{lemma}\label{lem:q-in-range}
There exists a constant $C>0$ such that the following holds for all $y \in [0,1]$: If $\ell >C\sqrt{n} (\log n)^2$ and $k \geq yn$ are integers, then $q_n(nye^{1-y}-\ell,k) \leq 2n^{-(\log \log n)/2}.$
\end{lemma}

\begin{proof}
This follows from Lemma~\ref{lem:curve-concentration} and the observation that $y \leq y'$ implies that $L_{\mathsf{runsort}(\pi)}(y) \leq L_{\mathsf{runsort}(\pi)}(y')$ for every $\pi$.
\end{proof}

\section{The Interior Density}\label{sec:interior}
We now compute the limiting density for the non-singular part of the permuton.  Suppose that $\pi$ is a random permutation of length $n$.  Recall that $p_n(i,j)$ denotes the probability that the entry $j$ is in the $i$-th position of $\mathsf{runsort}(\pi)$.  In the previous section, we analyzed the case where the entry $j$ is the beginning of a run of $\pi$, so we focus on the remaining case: Let $p'_n(i,j)=p_n(i,j)-q_n(i,j)$ denote the probability that the entry $j$ is in the $i$-th position of $\mathsf{runsort}(\pi)$ and is not the beginning of a run of $\pi$.  We will be interested in the situation where $i=xn$ and $j=yn$ for fixed $x,y \in (0,1)$, where the point $(x,y)$ lies in $\mathcal{C}^+$ (i.e., strictly above $\mathcal{C}$) and $n$ tends to infinity.  By the results of the previous section, we know that $p'_n(i,j)$ is very close to $p_n(i,j)$ in this regime. Whenever we use $xn$ or $yn$ as an input for a function that takes integer values (such as $p_n$, $p_n'$, or $q_n$), we really mean $\left\lceil xn\right\rceil$ and $\left\lceil yn\right\rceil$; we simply omit the ceiling symbols to avoid a typhoon of ceiling symbols. 

Fix $n$ and $i,j \in [n]$ with $i >1$, and suppose that we obtain $\pi$ by first picking a random permutation $\pi'$ on $[n] \setminus \{j\}$ and then inserting the entry $j$ in a random position.  Let $k$ denote the entry in the $(i-1)$-th position of $\mathsf{runsort}(\pi)$.  The following two descriptions define the same event:
\begin{itemize}
\item The entry $j$ is in the $i$-th position of $\mathsf{runsort}(\pi)$ and is not the beginning of a run of $\pi$.
\item $k<j$, and the entry $j$ was inserted directly after the entry $k$ in $\pi$.
\end{itemize}
The probability of the first bullet point occurring is (by definition) $p'_n(i,j)$, and the probability of the second bullet point occurring is
$$\frac{1}{n}\sum_{1 \leq k <j}p_{n-1}(i-1,k)=\frac{1}{n}\left(1-\sum_{j \leq k \leq n-1}p_{n-1}(i-1,k) \right).$$
Putting these together, we find that
\begin{equation}\label{eqn:recurrence}
p'_n(i,j)=\frac{1}{n}\left(1-\sum_{j \leq k \leq n-1}p_{n-1}(i-1,k) \right).
\end{equation}
Since $p_n(i,j)$ and $p'_n(i,j)$ are very close, we incur a very small error if we replace $p_{n-1}(i-1,k)$ with $p'_{n-1}(i-1,k)$ on the right-hand side; we will address this carefully below.  So, up to this small error, $p'_n(i,j)$ is equal to
$$\frac{1}{n}\left(1-\sum_{j \leq k \leq n-1}p'_{n-1}(i-1,k) \right).$$
Note that we have the boundary conditions $p'_n(1,j)=0$ for all $j$ and $p'_n(i,n)=1/n$ for $1<i<n$.  We now define the quantities $\widetilde{p}_n(i,j)$ recursively via
$$\widetilde{p}_n(i,j)=\frac{1}{n}\left(1-\sum_{j \leq k \leq n-1}\widetilde{p}_{n-1}(i-1,k) \right),$$
with the same boundary conditions $\widetilde{p}_n(1,j)=0$ and $\widetilde{p}_n(i,n)=1/n$ for $1<i<n$.
We will see that $p'_n(xn,yn)$ and $\widetilde{p}_n(xn,yn)$ are very close as long as $n$ is sufficiently large (with respect to $(x,y)$).

By repeatedly applying the recurrence relation for $\widetilde{p}$, we can express $\widetilde{p}_n(i,j)$ as a sum of terms involving $\widetilde{p}_{n-i+1}(1,k)$ (for $j \leq k \leq n-i+1$) and $\widetilde{p}_{n-r}(i-r,n-r)$ (for $i-r>1$ and $j \leq k \leq n-r$), together with some constant terms:
\begin{itemize}
\item The terms $\widetilde{p}_{n-i+1}(1,k)$ vanish by our boundary conditions.
\item Each term $\widetilde{p}_{n-r}(i-r,n-r)$ is equal to $1/(n-r)$ (by our boundary conditions) and appears $\binom{n-j+r-1}{r-1}$ times (by stars and bars), always weighted by $\frac{1}{n(n-1) \cdots (n-r+1)}$ and carrying the sign $(-1)^r$.  Here, $r$ ranges from $1$ to $\min\{n-j,i-2\}$; call the latter quantity $M$.
\item Each constant term $\frac{(-1)^s}{n(n-1) \cdots (n-s)}$ appears $\binom{n-j}{s}$ times, where $s$ ranges from $0$ to $M$.
\end{itemize}

Putting everything together, we arrive at the explicit formula
$$\widetilde{p}_n(i,j)=\sum_{r=1}^M \frac{(-1)^r}{n(n-1) \cdots (n-r)}\binom{n-j+r-1}{r-1}+\sum_{s=0}^M \frac{(-1)^s}{n(n-1) \cdots (n-s)}\binom{n-j}{s}.$$
We remark that the only dependence on $i$ is contained in the value of $M$.  In the regime $i \geq n-j+2$, we see that $M$ (and hence also $\widetilde{p}_n(i,j)$) is completely independent of $i$; this is a hint of the horizontal uniformity alluded to in the Introduction.  For $i=xn$ and $j=yn$ with $n$ tending to infinity, the second sum (call it $S_2$) will contribute the main term, and the first sum (call it $S_1)$ will contribute a negligible error.  Note that in this setting, $M$ is asymptotically a positive constant multiple of $n$ (since $x,y\in(0,1)$).

We begin with the first sum.  Expanding the binomial coefficient gives
$$S_1=\sum_{r=1}^M \frac{(-1)^r}{(r-1)!}\cdot\frac{1}{n(n-1)} \cdot \frac{(n-j+r-1)(n-j+r-2) \cdots (n-j+1)}{(n-2)(n-3)\cdots (n-r)}.$$
Writing $n-j=(1-y)n$, we estimate
$$\frac{n-j+t}{n-r+t-1}=1-\frac{j}{n}+\frac{j(t-r-1)+n(r-1)}{n(n-r+t-1)}=(1-y)+O(r/n)$$
uniformly in $n,r,t$ (so the implied constant depends only on $(x,y)$).  Then
\begin{align*}S_1 &=\frac{1}{n(n-1)}\sum_{r=1}^M \frac{(-1)^r}{(r-1)!} \left((1-y)^{r-1}+O\left(r2^{r-1}/n \right)\right)\\
 &=\frac{-e^{y-1}}{n^2} +O(1/n^3),
\end{align*}
which is $o(1/n)$.

Performing the analogous computation for the second sum, we find that
\begin{align*}
S_2 &=\sum_{s=0}^M\frac{(-1)^s}{s!}\cdot\frac{1}{n}\cdot\frac{(n-j)(n-j-1)\cdots(n-j-s+1)}{(n-1)(n-2)\cdots (n-s)} \\ &=\frac{1}{n}\sum_{s=0}^M\frac{(-1)^s}{s!}((1-y)^s+O((s+1)2^{s+1}/n)) \\ &=\frac{e^{y-1}}{n}+O(1/n^2).    
\end{align*}

In summary, we have established the following proposition.

\begin{proposition}\label{prop:p-tilde}
Fix $x,y \in (0,1)$ such that $(x,y) \in \mathcal{C}^+$.  Then
$$\widetilde{p}_n(xn,yn)=\frac{e^{y-1}}{n}+O(1/n^2),$$
where the implicit constant depends only on $(x,y)$.
\end{proposition}
It remains to bound the difference between $\widetilde{p}$ and $p'$.  This consists of keeping track of the error terms $q=p-p'$ when we iterate the recurrence \eqref{eqn:recurrence}.  At the $\ell$-th stage, the number of terms $p_{n-\ell}$ is $\binom{(1-y)n}{\ell}$ (by stars and bars), and each such term is scaled (up to a sign) by $\frac{1}{n(n-1)\cdots (n-\ell+1)}$.  Let $Q$ be the maximum of the error terms $q$ (which are certainly nonnegative).  By the triangle inequality, we may ignore the signs of the errors, and we find that
\[|p'_n(xn,yn)-\widetilde{p}_n(xn,yn)| \leq Q \sum_{\ell=1}^{(1-y)n} \binom{(1-y)n}{\ell} \frac{1}{n(n-1)\cdots (n-\ell+1)}\leq Q \sum_{\ell=1}^{(1-y)n} \frac{1}{\ell!}=O(Q).\]
We now bound $Q$.
\begin{lemma}\label{lem:E-bound}
Fix $x,y \in (0,1)$ such that $(x,y) \in \mathcal{C}+$.  If $n$ is sufficiently large (depending on $(x,y)$), then the following holds: For all $1 \leq \ell \leq \min \{xn-1,(1-y)n\}$ and all $yn \leq k \leq n-\ell$, we have
$$q_{n-\ell}(xn-\ell,k) \leq 2 (yn)^{-(\log \log (yn))/2}.$$
\end{lemma}

\begin{proof}
We wish to apply Lemma~\ref{lem:q-in-range} with $n$ replaced by $n-\ell$.  Let $\varepsilon=ye^{1-y}-x$ (which is strictly positive).  First, we check that
\[k\geq yn\geq y(n-\ell).\]
Second, we wish to show that $xn-\ell<(n-\ell)ye^{1-y}-C\sqrt{n-\ell} (\log (n-\ell))^2$ (where $C$ is the constant from Lemma~\ref{lem:q-in-range}); this inequality rearranges to
$$\varepsilon n+\ell(1-ye^{1-y})>C \sqrt{n-\ell} (\log (n-\ell))^2.$$
The term $1-ye^{1-y}$ is nonnegative, and we see that the inequality holds as long as $n$ is sufficiently large (depending on $y$ and $\varepsilon$).  So we can apply Lemma~\ref{lem:q-in-range}, which tells us that
$$q_{n-\ell}(xn-\ell,k)\leq 2(n-\ell)^{-(\log\log(n-\ell))/2}.$$
The right-hand side is an increasing function of $\ell$, so the bound $\ell \leq (1-y)n$ gives the desired inequality.
\end{proof}
The previous lemma implies that $Q \leq 2 (yn)^{-(\log \log (yn))/2}$ (which is certainly $O(1/n^2)$) for $n$ sufficiently large, so we can deduce the main result of this section.  (For $y=1$, recall from above that $p'_n(xn,n)=1/n$.)

\begin{lemma}\label{lem:internal}
Fix $x,y \in (0,1]$ such that $(x,y) \in \mathcal{C}^+$.  Then
$$p'_n(xn,yn)=\frac{e^{y-1}}{n}+O(1/n^2),$$
where again the implied constant depends only on $(x,y)$.
\end{lemma}

Let us summarize in words what this lemma tells us about the contribution to the measures ${\bf R}_n$ (and eventually also ${\bf R}$) from the entries that are not the beginnings of runs: In $\mathcal{C}^+$, this contribution gives a density $e^{1-y} \,dx \,dy$ at the point $(x,y)$; note that this is independent of $x$.  We have said nothing about the contribution below $\mathcal{C}$; that this contribution is $0$ follows quickly from Lemma~\ref{lem:curve-concentration}, but in fact we will give an alternative argument in the next section.

\section{Putting Everything Together}\label{sec:everything}

We finally prove Theorem~\ref{thm:main-reduction}, which implies Theorem~\ref{thm:main}.
The main idea is that we have already accounted for 100\% of the mass of ${\bf R}$ in our discussions in the previous two sections; this means that there will not be any mass below the curve $\mathcal{C}$ and that we do not need to worry about additional contributions very close to $\mathcal{C}$ from entries that are not the beginnings of runs.

For each axis-parallel rectangle $B \subseteq [0,1]^2$, it follows from Lemmas~\ref{lem:singular-part} and~\ref{lem:internal} that 
\[\liminf_{n\to\infty} {\bf R}_n(B) \geq \int_{B\cap\,\mathcal C^+}e^{y-1}\,dx\,dy+\int_{B\cap\,\mathcal C}(1-y)\,dy={\bf R}(B).\]
Now fix an axis-parallel rectangle $B_1\subseteq [0,1]^2$. Choose a tiling of $[0,1]^2$ by axis-parallel rectangles $B_1,\ldots,B_5$ (such a tiling certainly exists). 
We have \[1=\limsup_{n\to\infty} {\bf R}_n([0,1]^2)=\sum_{a=1}^5\limsup_{n\to\infty}{\bf R}_n(B_a)\geq\sum_{a=1}^5\liminf_{n\to\infty}{\bf R}_n(B_a)\geq\sum_{a=1}^5{\bf R}(B_a)=
{\bf R}([0,1]^2)=1.\] 
These inequalities must all be equalities, so $\lim\limits_{n \to \infty}{\bf R}_n(B_1)$ exists and equals ${\bf R}(B_1)$. As $B_1$ was arbitrary, this completes the proof of Theorem~\ref{thm:main-reduction}.

\section{A Generalized Setting}\label{sec:conclusion}
In this brief concluding section, we mention a setting in which the ideas presented earlier---especially those concerning the concentration inequalities derived in Section~\ref{sec:enveloping}---still hold. 

The \emph{standardization} of a sequence of $n$ distinct integers is the permutation in $S_n$ that has the same relative order as the sequence. For example, the standardization of $4917$ is $2413$.
Let $\mathcal F\subseteq\bigcup_{n\geq 0}S_n$ be a family of permutations, and let $\mathcal F_n=\mathcal F\cap S_n$. Assume that every permutation obtained by taking the standardization of a prefix of a permutation in $\mathcal F$ is also in $\mathcal F$. Let us also assume that $\mathcal F$ contains the permutation $1\in S_1$ and that there is some constant $c>1$ such that $\lvert \mathcal F_n\rvert<n!/c^n$ for all sufficiently large $n$. 

We can use the family $\mathcal F$ to split an arbitrary permutation $\pi\in S_n$ into subsequences as follows. Let $\pi_{[a,b]}$ denote the subsequence of $\pi$ consisting of entries in positions $a,a+1,\ldots,b$, and let $\overline \pi_{[a,b]}$ be the standardization of $\pi_{[a,b]}$. Set $k_0=1$. Then let $k_1$ be the smallest integer that is greater than $k_0$ and satisfies $\overline \pi_{[k_0,k_1]}\not\in\mathcal F$; we make the convention that $k_1=n+1$ if $\pi\in\mathcal F$. If $k_1\neq n+1$, let $k_2$ be the smallest integer that is greater than $k_1$ and satisfies $\overline \pi_{[k_1,k_2]}\not\in\mathcal F$; we make the convention that $k_2=n+1$ if $\overline \pi_{[k_1,n]}\in\mathcal F$. Continue defining integers $k_i$ in this greedy fashion until reaching a step at which $k_r=n+1$. Note that $\overline \pi_{[k_0,k_1-1]},\overline \pi_{[k_1,k_2-1]},\ldots,\overline \pi_{[k_{r-1},k_r-1]}$ all belong to the family $\mathcal F$. Let us call the subsequences $\pi_{[k_0,k_1-1]},\pi_{[k_1,k_2-1]},\ldots,\pi_{[k_{r-1},k_r-1]}$ the \emph{$\mathcal F$-runs} of $\pi$. Define $\mathcal F\text{-}\mathsf{sort}(\pi)$ to be the permutation obtained by sorting the $\mathcal F$-runs of $\pi$ so that their minimal entries appear in increasing order. Note that $\mathcal F\text{-}\mathsf{sort}$ is the same as $\mathsf{runsort}$ when $\mathcal F$ is the family of increasing permutations (consisting of one permutation of each length).

\begin{figure}[ht]
  \begin{center}\includegraphics[height=5.7cm]{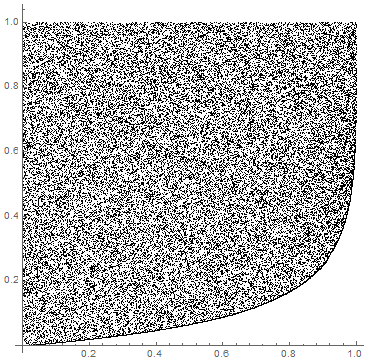}\qquad\qquad\qquad \includegraphics[height=5.7cm]{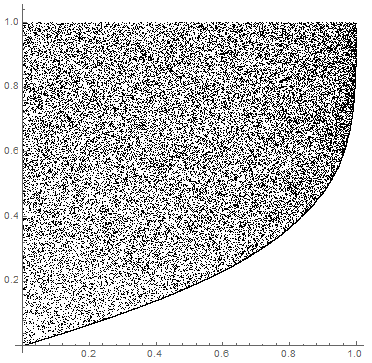}
  \end{center}
  \caption{The scaled plots of $\mathcal F^{\text{ddes}}\text{-}\mathsf{sort}(\pi)$ (left) and $\mathcal F^{\text{val}}\text{-}\mathsf{sort}(\pi)$ (right), where $\pi$ is a permutation chosen uniformly at random from $S_{50000}$.}\label{Fig2}
\end{figure}

Suppose $n$ is large. If we choose $\pi\in S_n$ uniformly at random, then the minimal entries of the $\mathcal F$-runs of $\pi$ should concentrate along a certain curve after we apply $\mathcal F\text{-}\mathsf{sort}$ to $\pi$. It should be possible to make this concentration statement precise using the ideas from Section~\ref{sec:enveloping}; however, determining exactly what the curve is could be very difficult. Here, we simply provide some images illustrating how this phenomenon might look in specific examples.

A \emph{double descent} of a permutation $\pi\in S_n$ is an index $i\in\{2,\ldots,n-1\}$ such that $\pi_{i-1}>\pi_i>\pi_{i+1}$. A \emph{valley} of $\pi$ is an index $i\in\{2,\ldots,n-1\}$ such that $\pi_{i-1}>\pi_{i}<\pi_{i+1}$. Let $\mathcal F^{\text{ddes}}$ be the family of permutations with no double descents, and let $\mathcal F^{\text{val}}$ be the family of permutations with no valleys. Figure~\ref{Fig2} shows the scaled plots of $\mathcal F^{\text{ddes}}\text{-}\mathsf{sort}(\pi)$ and $\mathcal F^{\text{val}}\text{-}\mathsf{sort}(\pi)$, where $\pi$ is a permutation chosen uniformly at random from $S_{50000}$.

One possibility for future research is the characterization of the limiting behavior of $\mathcal F\text{-}\mathsf{sort}(\pi)$ ($\pi \in S_n$ chosen uniformly at random) for various 
specific choices of $\mathcal{F}$.  Perhaps even more interesting would be the determination of more general properties such as the existence of a 
limiting permuton and conditions on $\mathcal{F}$ that guarantee some form of ``horizontal uniformity.''  One could also consider the limiting
behavior of $\mathsf{runsort}(\pi)$ when $\pi \in S_n$ is chosen non-uniformly, e.g., according to the Mallows distribution.

\section*{Acknowledgements}
The first author is supported in part by NSF grant DMS--1855464, BSF grant 2018267, and the Simons Foundation. The second author is supported by an NSF Graduate Research Fellowship (grant DGE--1656466) and a Fannie and John Hertz Foundation Fellowship.  The third author is supported by an NSF Graduate Research Fellowship (grant DGE--2039656). We are grateful to Ryan Alweiss and Peter Winkler for helpful conversations.

\end{document}